\documentclass{amsart}
\usepackage{amsmath}
\usepackage{amsthm}
\usepackage{amssymb}
\usepackage{hyperref}

\newtheorem{thm}{Theorem}
\newtheorem{lem}[thm]{Lemma}

\newtheorem{defi}[thm]{Definition}
\newtheorem{prop}[thm]{Proposition}
\newtheorem{rk}[thm]{Remark}

\newenvironment{preuve}{\vip\noindent {\it Proof}}{\hfill$\square$\vip}

\newcommand{\rr}{{\mathbb{R}}}
\newcommand{\e}{\epsilon}
\newcommand{\vip}{\vskip.2cm}
\newcommand{\indiq}{{\bf 1}}
\newcommand{\E}{\mathbb{E}}

\newcommand{\intot}{\int _0^t }
\newcommand{\intd}{\int_0^t \!\!\int_{\rr^3}}
\newcommand{\intdd}{\int_0^t \!\!\int_{\rr^3\times\rr^3}}

\newcommand{\tf}{{\tilde f}}
\newcommand{\tg}{{\tilde g}}
\newcommand{\tv}{{\tilde v}}
\newcommand{\tz}{{\tilde z}}
\newcommand{\tV}{{\tilde V}}

\newcommand{\cL}{{\mathcal L}}

\newcommand{\cW}{{\mathcal W}}
\newcommand{\cP}{{\mathcal P}}
\newcommand{\cH}{{\mathcal H}}
\newcommand{\cA}{{\mathcal A}}

\newcommand{\vs}{{v^*}}
\newcommand{\tvs}{{\tv^*}}

\begin{document}

\title[Landau equation with a Coulomb potential]
{Uniqueness of bounded solutions for the homogeneous Landau equation 
with a Coulomb potential}

\author{Nicolas Fournier}

\address{LAMA UMR 8050,
Facult\'e de Sciences et Technologies,
Universit\'e Paris Est, 61, avenue du G\'en\'eral de Gaulle, 94010 Cr\'eteil 
Cedex, France}

\email{nicolas.fournier@univ-paris12.fr}



\begin{abstract}
We prove the uniqueness of bounded solutions for the spatially homogeneous
Fokker-Planck-Landau equation with a Coulomb potential.
Since the local (in time) existence of such solutions has been
proved by Arsen'ev-Peskov \cite{ap}, we deduce a local well-posedness
result. The stability with respect to the initial condition is also checked.
\end{abstract}

\maketitle



\section{Introduction}


We consider the spatially homogeneous Landau
equation for a Coulomb potential. This equation 
of kinetic physics, also called Fokker-Planck-Landau equation,
has been derived from the Boltzmann equation
by Landau. It describes the 
density $f_t(v)$ of particles with velocity $v\in\rr^3$ at time 
$t\geq 0$ in a spatially homogeneous dilute plasma:
\begin{equation}\label{eqlandau}
\partial_t f_t(v) =\frac{1}{2}\sum_{i,j=1}^3
\partial_i \left(
\int_{\rr^3}a_{ij}( v-\vs) \Big[ f_t( \vs) \partial_j f_t( v) -f_t( v) 
\partial_j^* f_t( \vs) \Big]d\vs \right).
\end{equation}
Here $\partial_t =\frac{\partial }{\partial t}$, $\partial_i =
\frac{\partial }{\partial v_i}$, 
$\partial_i^* = \frac{\partial }{\partial v_i^*}$ and  for $z\in \rr^3$, 
$a(z)$ is the symmetric nonnegative 
matrix
\begin{equation} \label{defa}
a_{ij}(z)=|z|^{-3}(|z|^{2}\delta_{ij}-z_iz_j).
\end{equation} 
We refer to Villani \cite{v:nc,v:h}, Alexandre-Villani \cite{av} 
and the references therein for many information on this equation,
which has been widely used in plasma physics.
Let us mention that conservations of mass,
momentum and kinetic energy hold {\it a priori}, that is 
for $t\geq 0$, $\int \varphi(v)f_t(v) d v= \int \varphi(v)f_0(v)dv$,
for $\varphi(v)=1,v,|v|^2$.
We classically may assume without loss of generality that 
$\int f_0(v)dv=1$.
Another fundamental estimate, that we will not use here, 
is the decay of entropy: 
$\int f_t(v) \log f_t(v) dv \leq  \int f_0(v) \log f_0(v) dv$ for all $t\geq 0$.


Assume that in a dilute gas or plasma, particles collide by pairs, 
due to a repulsive force proportional to 
$1/r^s$, where $r$ stands for the distance
between the two particles. Then if $s\in (2,\infty)$, the 
velocity distribution solves the corresponding Boltzmann
equation \cite{v:nc,v:h}. But if $s=2$, the Boltzmann 
equation is meaningless \cite{v:nc} 
and is often replaced by the Landau equation
(\ref{eqlandau}). However, there are also many mathematical works on the 
Landau equation where $|z|^{-3}$ is replaced by
$|z|^{\gamma}$ in (\ref{defa}), with $\gamma=(s-5)/(s-1)\in [-3,1)$.
One usually speaks of hard potentials when $\gamma \in (0,1)$ (i.e. $s>5$), 
Maxwell molecules when $\gamma=0$ (i.e. $s=5$),
soft potentials when $\gamma \in (-3,0)$ (i.e.  $s \in (2,5)$), 
Coulomb potential when $\gamma=-3$ (i.e. $s=2$). 


When $\gamma>-3$, the Landau equation
can be seen as an approximation of the corresponding Boltzmann equation in the
asymptotic of {\it grazing collisions} \cite{v:nc}. This can help
to understand the effect of grazing collisions in the Boltzmann equation
without cutoff and is also of interest numerically.
But it seems that only the Landau equation with a Coulomb
potential has considerable importance in plasma physics.


The existence theory for the homogeneous Landau equation
is quite complete.
In \cite{v:nc}, Villani has proved the global
existence of weak solutions
to the homogeneous Landau equation for all possible potentials 
$\gamma\in[-3,1)$, for any initial
condition with finite mass, energy and entropy.
He also
showed in this paper that the solution to the Landau equation
can be seen as the limit of a sequence of solutions to some
suitable Boltzmann equations.
It is worth noting that for $\gamma\in [-2,1)$, 
the tools used in \cite{v:nc} are quite classical. But for
$\gamma \in [-3,-2)$, Villani uses some very fine
{\it a priori} estimates provided by the entropy dissipation.
The paper of Alexandre-Villani \cite{av} contains some existence results
in the much more difficult inhomogeneous case.


Uniqueness for the Landau equation is much less well-understood, 
even in the spatially homogeneous case.  
To our knowledge, this problem is still completely open in the 
realistic Coulomb case. Of course, uniqueness is very important, even
from the physical point of view: if uniqueness is not holding, this means
that the equation is not well-posed and thus that some additional
physical conditions have to be added.

Let us summarize the situations in which uniqueness for the 
homogeneous Landau equation
is known to hold. The initial condition $f_0$ is always supposed
to have finite mass and energy, $\int f_0(v)(1+|v|^2)dv<\infty$. 
On has global uniqueness when
$\gamma=0$, see Villani \cite{v:lmax},
when $\gamma \in (0,1)$ and $\int f_0^2(v) (1+|v|^q)]dv<\infty$ 
for some $q>5\gamma+15$, see Desvillettes-Villani \cite{dv},
and when $\gamma\in (-2,0]$ and
$\int [f_0(v)|v|^q +f_0(v)\log f_0(v)]dv<\infty$ for some
$q>\gamma^2/(2+\gamma)$, see \cite{fgl}.
One has local (in time) uniqueness when $\gamma\in (-3,-2]$ and 
$f_0 \in L^p$ for some $p>3/(3+\gamma)$, see \cite{fgl}.

The goal of this paper is to extend this final result to the case of a Coulomb
potential $\gamma=-3$, showing local uniqueness
for bounded initial conditions with finite mass and energy.
As a matter of fact, we will show that for any $T$, uniqueness holds
in the space $L^1([0,T],L^\infty(\rr^3))$. But the only known
result that provides existence of 
such solutions is that of Arsen'ev-Peskov \cite{ap}:
for $f_0$ bounded, one can find $T_*(f_0)>0$ and a solution
to (\ref{eqlandau}) lying to $L^\infty([0,T_*(f_0)]\times \rr^3)$. 
Thus our uniqueness result is not so satisfying at the moment,
since it concerns a functional space in which solutions
are known to belong only for bounded time interval.

The Coulomb case is really more difficult than the case
$\gamma>-3$, essentially because $|z|^{-3}$ is not integrable near $0$.
Thus while the global scheme of the proof is the same as in 
\cite{fgl}, the central computations are much more
delicate and {\it borderline}. 
It seems that Villani's existence results 
can be extended to the case $\gamma\in(-4,-3]$, see \cite[p 284]{v:nc},
but we are here really at the boundary of our possibilities.

Let us finally mention that the starting point of our proof is the famous
work of Tanaka \cite{t}, who proved the first uniqueness
result for the Boltzmann equation without cutoff (for Maxwell molecules).
We already used Tanaka's approach to study uniqueness for the 
Boltzmann equation without cutoff for hard and soft potentials \cite{fm,fgb}.

\section{Main result}

Let $\cP$ be the set of probability measures on $\rr^3$ and
$$\cP_2=\{f\in \cP, m_2(f)<\infty\}, \quad \hbox{ where } \quad
m_2(f)=\int_{\rr^3} |v|^2 f(dv).$$
For a measurable family $(f_t)_{t\in [0,T]} \subset \cP$,
we say that 
$$(f_t)_{t\in [0,T]}\in L^\infty([0,T],\cP_2) 
\quad \hbox{ if } \quad
\sup_{[0,T]} m_2(f_t)<\infty.
$$
Observe that any reasonable solution to (\ref{eqlandau}) belongs to
$L^\infty([0,T],\cP_2)$, because $m_2(f_t)=m_2(f_0)$.

When $f\in \cP$ has a bounded density, we say that $f\in L^\infty$,
we also denote by $f$ its density and by $||f||_\infty$
its $L^\infty$-norm.
For a measurable family $(f_t)_{t\in [0,T]} \subset \cP$,
we say that 
$$
(f_t)_{t\in [0,T]}\in L^1([0,T],L^\infty) \quad \hbox{ if } \quad
\int_0^T ||f_t||_\infty dt <\infty.
$$ 
We denote by $C^2_b$ the set of $C^2$ functions $\varphi:\rr^3\mapsto \rr$
with bounded derivatives of order $0$ to $2$.
For $\varphi\in C^2_b$ and $v,\vs\in \rr^3$, we introduce
\begin{align}\label{defL}
&L\varphi(v,\vs)=\frac{1}{2}\sum_{i,j=1}^3a_{ij}(v-\vs)\partial_{ij}^2\varphi(v)
+\sum_{i=1}^3 b_i(v-\vs)\partial_i\varphi(v) \\
\label{defb}
&\hbox{where}\quad  b_i(z)=\sum_{j=1}^3\partial_j a_{ij}(z)
=-2|z|^{-3} z_i.
\end{align} 
\begin{defi}\label{dfw}
We say that $(f_t)_{t\in[0,T]}$
is a weak solution to (\ref{eqlandau}) starting from $f_0\in \cP_2$ if
$(f_t)_{t\in[0,T]} \in L^\infty ([0,T],\mathcal{P}_2)\cap L^1([0,T],L^\infty)$
and if for any $\varphi \in\mathcal{C}^2_b$, any $t\in[0,T]$,
\begin{equation}\label{eqw}
\int_{\rr^3}\varphi(v)f_t(v)dv= \int_{\rr^3}\varphi(v)f_0(v)dv + \intot
\int_{\rr^3}\int_{\rr^3} f_t(v)f_t(\vs)L\varphi(v,\vs)dvd\vs d s.
\end{equation}
\end{defi}

For $\varphi \in C^2_b$, one has 
$|L\varphi(v,\vs)|\leq 
C_\varphi(|v-\vs|^{-1}+|v-\vs|^{-2})$ for some constant $C_\varphi$. 
Thus (\ref{a2}) and 
our conditions on $(f_t)_{t\in[0,T]}$ ensure us that
all the terms are well-defined in (\ref{eqw}).
The weak formulation (\ref{eqw}) is standard and can be found in 
\cite[Eq. (36)]{v:nc}.


We will widely use 
the Wasserstein distance $\cW_2$, defined
for $f,\tf \in \cP_2$, by
\begin{align*}
\cW_2^2 (f,\tf)=& \inf \left\{\E[|V-\tV|^2], \; V\sim f, \tV\sim \tf 
\right\} \\
=& \inf \left\{ \int_{\rr^3\times\rr^3} |v-\tv|^2 R(dv,d\tv), 
\; R \in \cH(f,\tf)\right\}.
\end{align*}
Here $V\sim f$ means that $V$ is a $\rr^3$-valued random variable
with law $f$ and $\cH(f,\tf)$ is the set of all 
probability measures on $\rr^3\times\rr^3$ with marginals $f$ and $\tf$.
The set $(\mathcal{P}_2,\cW_2)$ is a Polish space and its
topology is slightly stronger than the weak topology, 
see Villani \cite[Theorem 7.12]{v:t}.
It is well-known \cite[Chapter 1]{v:t}
that the infimum is reached: for $f,\tf\in \cP_2$,
we can find $R\in \cH(f,\tf)$ and $V\sim f, \tV\sim f$
such that $\cW_2^2 (f,\tf)= \int_{\rr^3\times\rr^3} |v-\tv|^2 R(dv,d\tv)
=\E[|V-\tV|^2]$.


Our main result reads as follows.

\begin{thm}\label{mr}
Let $T>0$. 

\noindent (i) For $f_0\in\cP_2$, there is at most one weak solution
to (\ref{eqlandau}) starting from $f_0$ and belonging to
$L^\infty ([0,T],\mathcal{P}_2)\cap L^1([0,T],L^\infty)$.

\noindent (ii) Assume that 
we have some weak solutions $(f_t)_{t\in[0,T]}$ and $(f^n_t)_{t\in [0,T]}$
to (\ref{eqlandau}), all
belonging to $L^\infty ([0,T],\mathcal{P}_2)\cap L^1([0,T],L^\infty)$.
If 
$\sup_n \int_0^T ||f^n_t||_\infty dt <\infty$
and
$\lim_n \cW_2(f^n_0,f_0)=0$, then $\lim_n \sup_{[0,T]} \cW_2(f^n_t,f_t) =0$.
\end{thm}

Arsen'ev-Peskov \cite{ap} have proved the following existence
result. Let $A>0$ be fixed. There exist some constants
$T_A>0$ and $C_A$ depending only on $A$ 
(with $\lim_{A\searrow 0} T_A = +\infty$) 
such that for any $f_0\in L^\infty \cap \cP_2$ with
$||f_0||_\infty \leq A$, 
there exists a weak solution
$(f_t)_{t\in [0,T_A]}$ to (\ref{eqlandau}) satisfying 
$\sup_{[0,T_A]} ||f_t||_\infty \leq C_A$. Theorem
\ref{mr} ensures us that this solution is unique and 
continuous with respect to the initial condition.
The result in \cite{ap} is based on the (formally easy) estimate
$\frac{d}{dt}||f_t||_\infty \leq C(1+||f_t||_\infty^2)$. 


In the next section, we establish
some fundamental regularity estimates
on the coefficients $a$ and $b$ of the Landau equation and we recall
a well-known generalization of the Gronwall Lemma.
The proof of Theorem \ref{mr} is handled in Sections \ref{prpr} and \ref{concl}.
In the whole paper, $C$ stands for a universal constant,
whose value changes from line to line.

\section{Preliminaries}

We will study the Landau equation through  
a stochastic differential equation whose
coefficients are $b$ (recall (\ref{defb})) and $\sigma$, defined 
for $z\in \rr^3$ by
\begin{equation}\label{defs}
\sigma \left( z\right) =\left| z\right| ^{\frac{-3 }{2}}\left( 
\begin{array}{ccc}
z_{2} & -z_{3} & 0 \\ 
-z_{1} & 0 & z_{3} \\ 
0 & z_{1} & -z_{2}
\end{array}
\right).
\end{equation}
For all $z\in \rr^3$, one has $\sigma(z).(\sigma(z))^t=a(z)$,
recall (\ref{defa}). 

\begin{lem}\label{estimsb}
For any $z,\tilde z\in \rr^3$,
\begin{align*}
|\sigma(z)-\sigma(\tilde z)|^2\leq& C \min
\left\{ |z-\tilde z|^2 (|z|^{-3}+
|\tilde z|^{-3}) ; |z|^{-1}+|\tz|^{-1}\right\}, \\
|b(z)-b(\tilde z)| \leq&  C\min 
\left\{ |z-\tilde z|(|z|^{-3}+|\tilde z|^{-3}) ; |z|^{-2}+|\tz|^{-2}\right\}.
\end{align*}
\end{lem}

\begin{proof}
First, we have $|\sigma(z)|\leq |z|^{-1/2}$, whence
$|\sigma(z)-\sigma(\tilde z)|^2\leq 2(|z|^{-1}+|\tz|^{-1})$.
Next,
\begin{align*}
|\sigma(z)-\sigma(\tilde z)|&\leq\left||z|^{-3/2}-|\tilde z|^{-3/2}\right|.|z|
+|z-\tilde z|.|\tilde z|^{-3/2}\\
&\leq \frac{3}{2} |z|.|z-\tilde z|
\max(|z|^{-5/2},|\tilde z|^{-5/2})
+|z-\tilde z|(|z|^{-3/2}+|\tz|^{-3/2}).
\end{align*}
By symmetry, we deduce that
\begin{align*}
|\sigma(z)-\sigma(\tilde z)| 
&\leq |z-\tilde z|\left(\frac{3}{2} 
\min(|z|,|\tz|) \max(|z|^{-5/2},|\tilde z|^{-5/2})
+|z|^{-3/2}+|\tz|^{-3/2}\right)\\
&\leq\frac{5}{2} |z-\tilde z|
\left(|z|^{-3/2}+|\tilde z|^{-3/2}\right).
\end{align*}
We also have $|b(z)|\leq 2|z|^{-2}$, so that
$|b(z)-b(\tz)| \leq 2 (|z|^{-2}+|\tz|^{-2})$. Finally,
\begin{align*}
|b(z)-b(\tz)| &\leq 2\left||z|^{-3}-|\tilde z|^{-3}\right|.|z|
+2|z-\tilde z|.|\tilde z|^{-3}\\
&\leq 6 |z|\max(|z|^{-4},|\tz|^{-4})|z-\tz| + 2|z-\tilde z|(|z|^{-3}+
|\tz|^{-3}),
\end{align*}
whence by symmetry,
\begin{align*}
|b(z)-b(\tz)| &\leq 6 \min(|z|,|\tz|)\max(|z|^{-4},|\tz|^{-4})|z-\tz| + 
2|z-\tilde z|(|z|^{-3}+|\tz|^{-3})\\
&\leq 8 |z-\tilde z|(|z|^{-3}+|\tz|^{-3}),
\end{align*}
which ends the proof.
\end{proof}

Next, we state some easy estimates of constant use in the paper.

\begin{lem}
Let $\alpha \in (-3,0]$. There is a constant $C_\alpha$ such that
for all $g \in \cP \cap L^\infty$, all $\e \in (0,1]$,
\begin{align}
&\sup_{v\in \rr^3} \int_{\rr^3} |v-\vs|^\alpha g(\vs) d\vs 
\leq 1+C_\alpha ||g||_\infty, \label{a1}\\
&\int_{\rr^3}\int_{\rr^3} |v-\vs|^\alpha g(v)g(\vs)dvd\vs 
\leq 1+C_\alpha ||g||_\infty, \label{a2} \\
&\sup_{v,w\in \rr^3} \int_{|v-\vs|\leq \e} |w-\vs|^\alpha g(\vs)d\vs  
\leq C_\alpha||g||_\infty \e^{3+\alpha}.\label{a3}
\end{align}
There is a constant $C$ such that  
for all $g \in \cP \cap L^\infty$, all $\e \in (0,1]$,
\begin{equation}\label{a4}
\sup_{v\in \rr^3}\int_{|v-\vs|\geq \e} |v-\vs|^{-3}g(\vs)d\vs
\leq 1+ C ||g||_\infty \log(1/\e).
\end{equation}
\end{lem}

\begin{proof}
Below, the variable $u$ belongs to $\rr^3$.
Recall that there is a constant $C$ such that 
for $\e \in (0,1]$,
\begin{equation}\label{cc1}
\int_{\e \leq |u|\leq 1} |u|^{-3} du = C \log(1/\e) 
\end{equation}
and that for $\alpha \in (-3,0]$, there is a constant
$C_\alpha$ such that for all $\e\in (0,1]$, all $u_0\in \rr^3$,
\begin{equation}\label{cc2}
\int_{|u|\leq \e} |u+u_0|^{\alpha} du
\leq  C_\alpha \e^{3+\alpha}.
\end{equation}
Since $g$ has mass $1$ and $\alpha\in (-3,0]$, 
for any $v\in \rr^3$,
\begin{align*}
\int_{\rr^3}|v-\vs|^\alpha g(\vs)d\vs &\leq \int_{|v-\vs|\geq 1}
g(\vs)d\vs+ \int_{|v-\vs|\leq 1}
|v-\vs|^\alpha g(\vs)d\vs \\
&\leq  1+ ||g||_\infty  \int_{|v-\vs|\leq 1} |v-\vs|^\alpha d\vs\\
&= 1+ ||g||_\infty  \int_{|u|\leq 1} |u|^\alpha du,
\end{align*}
whence (\ref{a1}) due to (\ref{cc2}) with $\e=1$. Inequality 
(\ref{a2}) follows from (\ref{a1}) because $g$ has mass
$1$. Next (\ref{a3}) is deduced from (\ref{cc2}): for $v,w\in \rr^3$
and $\e\in (0,1]$,
\begin{align*}
\int_{|v-\vs|\leq \e} |w-\vs|^\alpha g(\vs)d\vs \leq&
||g||_\infty \int_{|v-\vs|\leq \e} |w-\vs|^\alpha d\vs \\
=&||g||_\infty \int_{|u|\leq \e} |u+(v-w)|^\alpha du.
\end{align*}
Finally, for any $v\in \rr^3$,
\begin{align*}
\int_{|v-\vs|\geq \e} |v-\vs|^{-3} g(\vs)d\vs \leq& 
\int_{|v-\vs|\geq 1} g(\vs)d\vs
+ ||g||_\infty \int_{\e\leq |v-\vs|\leq 1} |v-\vs|^{-3} d\vs \\
\leq&1+||g||_\infty  \int_{\e \leq |u|\leq 1} |u|^{-3} du,
\end{align*}
from which (\ref{a4}) follows using (\ref{cc1}).
\end{proof}

We also consider the increasing continuous
function $\Psi:[0,\infty)\mapsto \rr_+$ defined by
\begin{equation}\label{dfpsi}
\Psi(x)=x(1 - \indiq_{\{x\leq 1\}}\log x).
\end{equation}
The following remark will allow us to apply the Jensen
inequality.
\begin{rk}\label{conc}
We can find a concave increasing continuous 
function $\Gamma:\rr_+\mapsto \rr_+$ such that for all
$x\geq 0$, $ \Psi(x)/2 \leq \Gamma(x)\leq 2 \Psi(x)$.  
\end{rk}
\begin{proof}
Choose $\Gamma(x)=x(1-\log x)$ for $x \in [0,1/2]$ and
$\Gamma(x)=x\log 2 + 1/2$ for $x\geq 1/2$.
\end{proof}

The two next lemmas contain the fundamental computations
of the paper.

\begin{lem}\label{toutestla}
For any $g \in \cP\cap L^\infty$, for all $v,\tv\in\rr^3$,
\begin{align}
&\int_{\rr^3} |\sigma(v-\vs)-\sigma(\tv-\vs)|^{2}g(\vs)d\vs 
\leq C(1+||g||_\infty) \Psi(|v-\tv|^2),  \label{ttl1}\\
&\int_{\rr^3} |b(v-\vs)-b(\tv-\vs)|g(\vs)d\vs \leq C(1+||g||_\infty)
\Psi(|v-\tv|). \label{ttl2}
\end{align}
\end{lem}

\begin{proof}
We denote by $I$ the left hand side of (\ref{ttl1}). 
Using Lemma \ref{estimsb},
$$
|\sigma(v-\vs)-\sigma(\tv-\vs)|^{2}\leq C\min\left\{ 
|v-\tv|^2(|v-\vs|^{-3}+|\tv-\vs|^{-3});|v-\vs|^{-1}+|\tv-\vs|^{-1}
\right\}. 
$$
Thus
\begin{align*}
I \leq&
C\indiq_{\{|v-\tv|\geq 1\}} \int_{\rr^3} (|v-\vs|^{-1}+|\tv-\vs|^{-1})g(\vs)d\vs\\
&+ C\indiq_{\{|v-\tv|\leq 1\}} \int_{\rr^3} \indiq_{\{|v-\vs|\geq |v-\tv|^2,
|\tv-\vs|\geq |v-\tv|^2 \}} |v-\tv|^2  (|v-\vs|^{-3}+|\tv-\vs|^{-3})g(\vs)d\vs\\
&+ C\indiq_{\{|v-\tv|\leq 1\}} \int_{\rr^3} \indiq_{\{|v-\vs|\leq |v-\tv|^2\}} 
(|v-\vs|^{-1}+|\tv-\vs|^{-1})g(\vs)d\vs\\
&+ C\indiq_{\{|v-\tv|\leq 1\}} \int_{\rr^3} \indiq_{\{|\tv-\vs|\leq |v-\tv|^2\}} 
(|v-\vs|^{-1}+|\tv-\vs|^{-1})g(\vs)d\vs\\
=:& C(I_1+I_2+I_3+I_4).
\end{align*}
First, (\ref{a1}) with $\alpha=-1$ implies that 
$$
I_1\leq C(1+||g||_\infty)
\indiq_{\{|v-\tv|\geq 1\}} \leq C(1+||g||_\infty)\Psi (|v-\tv|^2).
$$
Next, using (\ref{a4}) with $\e=|v-\tv|^2$ yields
\begin{align*}
I_2 \leq&\indiq_{\{|v-\tv|\leq 1\}} |v-\tv|^2 \Big(\int_{|v-\vs|\geq |v-\tv|^2}
|v-\vs|^{-3}g(\vs)d\vs + \int_{|\tv-\vs|\geq |v-\tv|^2} |\tv-\vs|^{-3}g(\vs)d\vs
\Big)\\
\leq& \indiq_{\{|v-\tv|\leq 1\}} |v-\tv|^2 (2+C ||g||_\infty \log(1/|v-\tv|^2))\\
\leq& C (1+||g||_\infty)\indiq_{\{|v-\tv|\leq 1\}} |v-\tv|^2(1- \log(|v-\tv|^2))\\
\leq& C(1+||g||_\infty)\Psi(|v-\tv|^2).
\end{align*}
Finally, we deduce from  (\ref{a3}) with $\alpha=-1$ and $\e=|v-\tv|^2$ that
\begin{align*}
I_3+I_4 \leq& C ||g||_\infty \indiq_{\{|v-\tv|\leq 1\}} 
(|v-\tv|^2)^{3-1}
\leq C ||g||_\infty |v-\tv|^2 
\leq C ||g||_\infty \Psi(|v-\tv|^2).
\end{align*}
We now denote by $J$ the left hand side of (\ref{ttl2}). By Lemma \ref{estimsb},
$$
|b(v-\vs)-b(\tv-\vs)|\leq C \min\left\{|v-\tv|(|v-\vs|^{-3}+|\tv-\vs|^{-3});
|v-\vs|^{-2}+|\tv-\vs|^{-2} \right\}.
$$
Thus
\begin{align*}
J \leq&
C\indiq_{\{|v-\tv|\geq 1\}} \int_{\rr^3} (|v-\vs|^{-2}+|\tv-\vs|^{-2})g(\vs)d\vs\\
&+ C\indiq_{\{|v-\tv|\leq 1\}} \int_{\rr^3} \indiq_{\{|v-\vs|\geq |v-\tv|^2,
|\tv-\vs|\geq |v-\tv|^2 \}}  |v-\tv| (|v-\vs|^{-3}+|\tv-\vs|^{-3})g(\vs)d\vs\\
&+ C\indiq_{\{|v-\tv|\leq 1\}} \int_{\rr^3} \indiq_{\{|v-\vs|\leq |v-\tv|^2\}} 
(|v-\vs|^{-2}+|\tv-\vs|^{-2})g(\vs)d\vs\\
&+ C\indiq_{\{|v-\tv|\leq 1\}} \int_{\rr^3} \indiq_{\{|\tv-\vs|\leq |v-\tv|^2\}} 
(|v-\vs|^{-2}+|\tv-\vs|^{-2})g(\vs)d\vs\\
=:& C(J_1+J_2+J_3+J_4).
\end{align*}
Using (\ref{a1}) with $\alpha=-2$, we get
$$
J_1\leq C(1+||g||_\infty)
\indiq_{\{|v-\tv|\geq 1\}} \leq C(1+||g||_\infty)\Psi (|v-\tv|).
$$
Next, (\ref{a4}) with $\e=|v-\tv|^2$ yields 
\begin{align*}
J_2 \leq& 
\indiq_{\{|v-\tv|\leq 1\}} \left( 
\int_{|v-\vs|\geq |v-\tv|^2} |v-\vs|^{-3} g(\vs)d\vs +
\int_{|\tv-\vs|\geq |v-\tv|^2} |\tv-\vs|^{-3} g(\vs)d\vs  \right)\\
\leq&\indiq_{\{|v-\tv|\leq 1\}} |v-\tv| [2+C||g||_\infty \log(1/|v-\tv|^2)]\\
\leq& C    (1+||g||_\infty)\indiq_{\{|v-\tv|\leq 1\}} |v-\tv|[1- \log(|v-\tv|)]
\\
\leq& C (1+||g||_\infty) \Psi(|v-\tv|).
\end{align*}
Finally,
$J_3+J_4 \leq C ||g||_\infty \indiq_{\{|v-\tv|\leq 1\}} 
(|v-\tv|^2)^{3-2}  
\leq C ||g||_\infty \Psi(|v-\tv|)$
by (\ref{a3}) with $\alpha=-2$ and $\e=|v-\tv|^2$.
\end{proof}

\begin{lem}\label{toutestla2}
Consider $g,\tg \in \cP_2\cap L^\infty$ and $Q,R\in \cH(g,\tg)$. Then 
\begin{align}
&\int_{\rr^3\times\rr^3} \int_{\rr^3\times\rr^3}
|v-\tv|.|b(v-\vs)-b(\tv-\tvs)| Q(dv,d\tv)R(d\vs,d\tvs)\label{ttl3} \\
\leq& C(1+||g+\tg ||_\infty) \Big\{\Psi\left(\int_{\rr^3\times\rr^3}
|v-\tv|^2 Q(dv,d\tv)\right) + \Psi\left(\int_{\rr^3\times\rr^3}
|\vs-\tvs|^2 R(d\vs,d\tvs)\right)
\Big\}. \nonumber
\end{align}
\end{lem}

\begin{proof}
We denote by $K$
the left hand side of (\ref{ttl3}) and by $\delta(v,\tv,\vs,\tvs)=
|v-\tv|.|b(v-\vs)-b(\tv-\tvs)|$. Due to Lemma
\ref{estimsb}, $\delta$ is smaller than
\begin{align*}
C(|v-\tv|+|\vs-\tvs|)\min\Big\{(|v-\tv|+|\vs-\tvs|)(|v-\vs|^{-3} +
|\tv-\tvs|^{-3}); |v-\vs|^{-2}+|\tv-\tvs|^{-2}\Big\}.
\end{align*}
Hence we can write
\begin{align*} 
\delta(v,\tv,\vs,\tvs)
\leq& C\indiq_{\{|v-\tv|+|\vs-\tvs|\geq 1\}} (|v-\tv|+|\vs-\tvs|)
(|v-\vs|^{-2}+|\tv-\tvs|^{-2}) \\
+& C\indiq_{\{|v-\tv|+|\vs-\tvs|\leq 1\}}\indiq_{\{|v-\vs|\geq|v-\tv|^4,
|\tv-\tvs|\geq|v-\tv|^4\}} |v-\tv|^2 (|v-\vs|^{-3}
+ |\tv-\tvs|^{-3})\\
+& C\indiq_{\{|v-\tv|+|\vs-\tvs|\leq 1\}}\indiq_{\{|v-\vs|\geq|v-\tv|^4,
|\tv-\tvs|\geq|v-\tv|^4\}} |\vs-\tvs|^2 (|v-\vs|^{-3}
+ |\tv-\tvs|^{-3})\\
+& C\indiq_{\{|v-\tv|+|\vs-\tvs|\leq 1\}}\indiq_{\{|v-\vs|\leq|v-\tv|^4\}}
(|v-\tv|+|\vs-\tvs|)(|v-\vs|^{-2}+|\tv-\tvs|^{-2})\\
+&C\indiq_{\{|v-\tv|+|\vs-\tvs|\leq 1\}}\indiq_{\{|\tv-\tvs|\leq|v-\tv|^4\}}
(|v-\tv|+|\vs-\tvs|)(|v-\vs|^{-2}+|\tv-\tvs|^{-2})\\
=:&C\sum_1^5\delta_i(v,\tv,\vs,\tvs).
\end{align*}
Thus $K\leq C \sum_1^5 K_i$, where 
$K_i=\int_{\rr^3\times\rr^3}\int_{\rr^3\times\rr^3}  
\delta_i(v,\tv,\vs,\tvs)Q(dv,d\tv)
R(d\vs,d\tvs)$. First,
\begin{align*}
\delta_1(v,\tv,\vs,\tvs)\leq&
(|v-\tv|+|\vs-\tvs|)^2(|v-\vs|^{-2}+|\tv-\tvs|^{-2})\\
 \leq & 2(|v-\tv|^2+|\vs-\tvs|^2)(|v-\vs|^{-2}+|\tv-\tvs|^{-2}).
\end{align*}
As a consequence,
\begin{align*}
K_1 \leq&  2\int_{\rr^3\times\rr^3} |v-\tv|^2 Q(dv,d\tv) 
\int_{\rr^3\times \rr^3} (|v-\vs|^{-2}+ |\tv-\tvs|^2) R(d\vs,d\tvs) \\
&+ 2\int_{\rr^3\times\rr^3} |\vs-\tvs|^2 R(d\vs,d\tvs) 
\int_{\rr^3\times\rr^3} (|v-\vs|^{-2}+|\tv-\tvs|^{-2})Q(dv,d\tv).\\
=:& 2K_{1,1}+2K_{1,2}.
\end{align*}
Since now $R$ has marginals $g$ and $\tg$,
we deduce from (\ref{a1}) with $\alpha=-2$ that
\begin{align*}
K_{1,1} =& \int_{\rr^3\times\rr^3} |v-\tv|^2 Q(dv,d\tv) 
\left(\int_{\rr^3} |v-\vs|^{-2}g(\vs)d\vs + 
\int_{\rr^3} |\tv-\tvs|^2 \tg(\tvs)d\tvs \right) \\
\leq& \int_{\rr^3\times\rr^3} |v-\tv|^2 Q(dv,d\tv) 
[1+C ||g||_\infty + 1+C||\tg||_\infty]    \\
\leq&  C(1+||g+\tg||_\infty )
\int_{\rr^3\times\rr^3} |v-\tv|^2 Q(dv,d\tv) \\
\leq& C(1+||g+\tg||_\infty )\Psi\left(
\int_{\rr^3\times\rr^3} |v-\tv|^2 Q(dv,d\tv)\right).
\end{align*}
Similarly,
\begin{align*}
K_{1,2} \leq  C(1+||g+\tg||_\infty )\Psi\left(
\int_{\rr^3\times\rr^3} |\vs-\tvs|^2 R(d\vs,d\tvs)\right).
\end{align*}
Next,
\begin{equation*}
\delta_2(v,\tv,\vs,\tvs)\leq \indiq_{\{|v-\tv|\leq 1\}}
\indiq_{\{|v-\vs|\geq |v-\tv|^4,|\tv-\tvs| \geq |v-\tv|^4 \}}
|v-\tv|^2(|v-\vs|^{-3}+|\tv-\tvs|^{-3}).
\end{equation*}
Thus, (\ref{a4}) with $\e=|v-\tv|^4$ yields
\begin{align*}
K_2\leq & \int_{\rr^3\times\rr^3} Q(dv,d\tv) \indiq_{\{|v-\tv|\leq 1\}} 
|v-\tv|^2 \int_{\rr^3\times\rr^3}R(d\vs,d\tvs) \\
&\hskip2cm 
\left(\indiq_{\{|v-\vs|\geq |v-\tv|^4\}}|v-\vs|^{-3} 
+ \indiq_{\{|\tv-\tvs| \geq |v-\tv|^4\}}|\tv-\tvs|^{-3}\right) \\
= & \int_{\rr^3\times\rr^3} Q(dv,d\tv) \indiq_{\{|v-\tv|\leq 1\}} 
|v-\tv|^2 \\
&\hskip1cm 
\left(\int_{|v-\vs|\geq |v-\tv|^4}|v-\vs|^{-3}g(\vs)d\vs 
+ \int_{|\tv-\tvs|
\geq |v-\tv|^4}|\tv-\tvs|^{-3}\tg(\tvs)d\tvs\right) \\
\leq&
 \int_{\rr^3\times\rr^3} Q(dv,d\tv) \indiq_{\{|v-\tv|\leq 1\}} 
|v-\tv|^2 \left[1+C||g||_\infty \log(1/|v-\tv|^4) 
+1+C||\tg||_\infty \log(1/|v-\tv|^4)\right]\\
\leq& C(1+||g+\tg||_\infty) 
\int_{\rr^3\times\rr^3} Q(dv,d\tv) \indiq_{\{|v-\tv|\leq 1\}} 
|v-\tv|^2 [1+\log(1/|v-\tv|^2)]\\
\leq& C(1+||g+\tg||_\infty) \int_{\rr^3\times\rr^3} 
\Psi(|v-\tv|^2)Q(dv,d\tv).
\end{align*}
Remark \ref{conc} and the Jensen inequality allow us to conclude that
\begin{equation*}
K_2 \leq
C (1+||g+\tg||_\infty) \Psi\left(\int_{\rr^3\times\rr^3} |v-\tv|^2 
Q(dv,d\tv)\right).
\end{equation*}
The third term $K_3$ is bounded symmetrically.
For the fourth term, we first notice that
\begin{equation*}
\delta_4(v,\tv,\vs,\tvs)\leq  
\indiq_{\{|v-\tv|\leq 1\}} \indiq_{\{|v-\vs|\leq
|v-\tv|^4\}} [|v-\vs|^{-2}+|\tv-\tvs|^{-2}],
\end{equation*}
whence
\begin{equation*}
K_4 \leq  
\int_{\rr^3\times\rr^3} Q(dv,d\tv) \indiq_{\{|v-\tv|\leq 1\}} 
\int_{|v-\vs|\leq |v-\tv|^4} [|v-\vs|^{-2}+|\tv-\tvs|^{-2}] R(d\vs,d\tvs).
\end{equation*}
But for all $v,\tv$, using (\ref{a3}) with $\alpha=-2$ and
$\e=|v-\tv|^4$,
\begin{equation*}
\int_{|v-\vs|\leq |v-\tv|^4} |v-\vs|^{-2} R(d\vs,d\tvs)
=\int_{|v-\vs|\leq |v-\tv|^4} |v-\vs|^{-2} g(\vs)d\vs
\leq C||g||_{\infty}|v-\tv|^4.
\end{equation*}
Using now the H\"older inequality (with $p=5$ and $q=5/4$),
then (\ref{a3}) with $\alpha=0$, $\e=|v-\tv|^4$ and finally (\ref{a1})
with $\alpha=-5/2$,
\begin{align*}
&\int_{|v-\vs|\leq |v-\tv|^4} |\tv-\tvs|^{-2} R(d\vs,d\tvs) \\
\leq& \left( \int_{|v-\vs|\leq |v-\tv|^4} R(d\vs,d\tvs)\right)^{1/5}  
\left( \int_{\rr^3\times\rr^3} |\tv-\tvs|^{-5/2} R(d\vs,d\tvs)\right)^{4/5}\\
=&\left( \int_{|v-\vs|\leq |v-\tv|^4} g(\vs)d\vs \right)^{1/5}  
\left( \int_{\rr^3} |\tv-\tvs|^{-5/2} \tg(\tvs)d\tvs \right)^{4/5}\\
\leq& \left( C ||g||_\infty [|v-\tv|^4]^3 \right)^{1/5}(1+C||\tg||_\infty)^{4/5} 
\\
\leq& C(1+||g+\tg||_\infty) |v-\tv|^{12/5},
\end{align*}
because $||g||_\infty^{1/5}(1+||\tg||_\infty)^{4/5}  \leq 1+||g||_\infty
+||\tg||_\infty \leq 1+2||g+\tg||_\infty$.  We thus have shown that
\begin{align*}
K_4\leq&C (1+||g+\tg||_\infty) 
\int_{\rr^3\times\rr^3} Q(dv,d\tv) \indiq_{\{|v-\tv|\leq 1\}}
(|v-\tv|^4 +  |v-\tv|^{12/5})\\
\leq&  C (1+||g+\tg||_\infty) 
\int_{\rr^3\times\rr^3} Q(dv,d\tv) |v-\tv|^2\\
\leq&C (1+||g+\tg||_\infty) \Psi\left(\int_{\rr^3\times\rr^3} |v-\tv|^2 
Q(dv,d\tv)\right).
\end{align*}
The last term $K_5$ is treated symmetrically, which ends the proof.
\end{proof}

We conclude this section by recalling a generalization of the Gronwall Lemma
of which the proof can be found in Chemin \cite[Lemme 5.2.1 page 89]{c}.

\begin{lem}\label{yudo}
Let $T>0$ and $\gamma:[0,T]\mapsto \rr_+$ satisfy $\int_0^T \gamma(s)ds 
<\infty$. Recall that $\Psi$ was defined by (\ref{dfpsi}) and set
$M(x)=\int_x^1 (1/\Psi(y))dy$ for $x>0$. Consider a
bounded measurable function $\rho:[0,T]\mapsto \rr_+$ such that,
for some $a\geq 0$, for all $t\in [0,T]$,
$\rho(t) \leq a + \int_0^t \gamma(s)\Psi(\rho(s))ds$.

\noindent (i) If $a=0$, then $\rho(t)=0$ for all $t\in [0,T]$.

\noindent (ii) If $a>0$, then $M(a)-M(\rho(t)) \leq \int_0^t \gamma(s)ds$ 
for all $t\in [0,T]$.
\end{lem}

\section{An integral inequality}\label{prpr}

Theorem \ref{mr} will be easily deduced, in the next section,
from Lemma \ref{yudo} and the following result.
Recall that $\Psi$ was defined in (\ref{dfpsi}).

\begin{thm}\label{fonda}
There is a constant $C$ such that
for any pair of weak solutions $(f_t)_{t\in [0,T]}$ and
$(\tf_t)_{t\in [0,T]}$ to (\ref{eqlandau}), there is a bounded
function $\rho:[0,T]\mapsto \rr_+$ satisfying, for all $t\in [0,T]$,
\begin{equation*}
\cW^2_2(f_t,\tf_t)\leq \rho(t) \quad \hbox{  and } \quad
\rho(t) \leq \cW^2_2(f_0,\tf_0)
+C\intot (1+||f_s+\tf_s||_\infty) \Psi(\rho(s))ds.
\end{equation*}
\end{thm}

From now on, $T>0$ and the two weak solutions $(f_t)_{t\in [0,T]},
(\tf_t)_{t\in [0,T]}$ to (\ref{eqlandau}), both belonging to 
$L^\infty\big([0,T],\mathcal{P}_2\big)\cap L^1([0,T],L^\infty)$ are fixed.
We follow closely the scheme of proof of \cite{fgl}:
first, we introduce two coupled Landau stochastic processes,
the first one associated with $f$, the second one associated with 
$\tf$, in such a way
that they remain as close to each other as possible. 
The probabilistic interpretation of the Landau equation
we use here
has been introduced by Funaki \cite{f},
Gu\'erin \cite{g} and
is inspired by the work of Tanaka \cite{t}.


For all $s \in[0,T]$, we denote by $R_s \in \cH(f_s,\tf_s)$ 
the (unique) probability measure on $\rr^3\times\rr^3$
with marginals $f_s$ and $\tf_s$ such that 
$\cW_2^2(f_s,\tf_s)=\int_{\rr^3\times\rr^3} |v-\tv|^2 R_s(dv,d\tv)$.

On some probability space,
we consider a three-dimensional 
white noise $W(dv,d\tv,ds)$ on
$\rr^3\times \rr^3\times[0,T]$ with covariance measure 
$R_s(dv,d\tilde v)ds$.
This means that $W=(W_1,W_2,W_3)$, where 
$W_1$, $W_2$ and $W_3$ are three independent
white noises on $\rr^3\times \rr^3\times[0,T]$ with covariance measure 
$R_s(dv,d\tilde v)ds$, see Walsh \cite{w} for definitions.
We also need two $\rr^3$-valued random variables $V_0,\tilde V_0$ 
with laws $f_0,\tf_0$, independent of $W$, such that 
$\cW_2^2(f_0,\tilde f_0)=\E[|V_0-\tilde V_0|^2]$.
We consider the two following $\rr^3$-valued 
stochastic differential equations.
\begin{align}\label{eq:eds}
V_t&=V_0+\intdd \sigma(V_s-v)W(dv,d\tilde v,ds)+\intd b(V_s-v)f_s(v)dvds,\\
\tilde V_t&=\tilde V_0+\intdd \sigma(\tilde V_s-\tilde v)W(dv,d\tilde v,ds)
+\intd b(\tilde V_s-\tilde v)\tf_s(\tv)d\tv ds, \label{eq:edstilde}
\end{align}
$b,\sigma$ being defined by (\ref{defb}) and (\ref{defs}).
We set $\mathcal{F}_t=\sigma\{V_0,\tV_0,W([0,s]\times A), \;  
s\in[0,t], A \in {\mathcal B}(\rr^3\times\rr^3)\}$.

\begin{prop}\label{prop:edslin} 
(i) There exists a unique pair $(V_t)_{t\in [0,T]}$, $(\tV_t)_{t\in [0,T]}$
of continuous  $(\mathcal{F}_t)_{t\in [0,T]}$-adapted
processes solving (\ref{eq:eds}) and (\ref{eq:edstilde}).

\noindent (ii) For all $t\in [0,T]$, $\cL(V_t)=f_t$ and 
$\cL(\tV_t)=\tf_t$.
\end{prop}

We admit this proposition for a while.

\begin{preuve} {\it of Theorem \ref{fonda}.}
By Proposition
\ref{prop:edslin}-(ii), we have 
$\cW^2_2(f_t,\tf_t)\leq \E[|V_t-\tV_t|^2]$. 
We thus compute this last
quantity carefully.
The marginals of $R_s$ being $f_s$ and $\tf_s$, 
we may rewrite (\ref{eq:eds}) and (\ref{eq:edstilde}) as
\begin{align*}
V_t&=V_0+\intdd \sigma(V_s-v)W(dv,d\tilde v,ds)+\intdd b(V_s-v)R_s(dv,d\tv)ds,
\\
\tilde V_t&=\tilde V_0+\intdd \sigma(\tilde V_s-\tilde v)W(dv,d\tilde v,ds)
+\intdd b(\tilde V_s-\tilde v) R_s(dv,d\tv)ds.
\end{align*}
Using the It\^o formula and taking expectations, we obtain
\begin{align*}
\E[|V_t-\tilde V_t|^2]=&\E[|V_0-\tilde V_0|^2] +\sum_{i,l=1}^3 \intdd 
\E\big[ \big(\sigma_{il}(V_s-v)-\sigma_{il}(\tilde V_s-\tilde v)\big)^2\big]
R_s(dv,d\tilde v) ds \\
& +2\intdd \E\big[ \big(b(V_s-v)-b(\tilde V_s-\tilde v)\big).
(V_{s}-\tilde V_{s})\big]R_s(dv,d\tilde v) ds \\
=& \cW^2_2(f_0,\tf_0) +\int_0^t A_sds + 2 \int_0^t B_s ds.
\end{align*}
Let $Q_s(dv,d\tv)$ be the law of the couple $(V_s,\tV_s)$. 
Using (\ref{ttl3}) and that $R_s,Q_s\in\cH(f_s,\tf_s)$,
\begin{align*}
B_s=&\int_{\rr^3\times\rr^3} \int_{\rr^3\times\rr^3} 
|v-\tv|.|b(v-\vs)-b(\tv-\tvs)|Q_s(dv,d\tv) R_s(d\vs,d\tvs) \\
\leq& C(1+||f_s+\tf_s||_\infty)\left\{
\Psi\left(\int_{\rr^3\times\rr^3}|v-\tv|^2 Q_s(dv,d\tv)\right)
+\Psi\left(\int_{\rr^3\times\rr^3}|\vs-\tvs|^2 R_s(d\vs,d\tvs)\right)
\right\}.
\end{align*}
Next, using (\ref{ttl1}),
\begin{align*}
A_s =& \int_{\rr^3\times\rr^3} \int_{\rr^3\times\rr^3} 
|\sigma(v-\vs)-\sigma(\tv-\tvs)|^2 Q_s(dv,d\tv) R_s(d\vs,d\tvs) \\
\leq& 2 \int_{\rr^3\times\rr^3} \int_{\rr^3\times\rr^3} 
|\sigma(v-\vs)-\sigma(v-\tvs)|^2 Q_s(dv,d\tv) R_s(d\vs,d\tvs)\\
&+2 \int_{\rr^3\times\rr^3} \int_{\rr^3\times\rr^3} 
|\sigma(v-\tvs)-\sigma(\tv-\tvs)|^2 Q_s(dv,d\tv) R_s(d\vs,d\tvs) \\
=& 2  \int_{\rr^3\times \rr^3} \int_{\rr^3} 
|\sigma(v-\vs)-\sigma(v-\tvs)|^2 f_s(v)dv R_s(d\vs,d\tvs) \\
&+2 \int_{\rr^3\times\rr^3} \int_{\rr^3} 
|\sigma(v-\tvs)-\sigma(\tv-\tvs)|^2 \tf_s(d\tvs)Q_s(dv,d\tv)\\
\leq &C(1+||f_s||_\infty) 
\int_{\rr^3\times\rr^3}\!\! \Psi(|\vs-\tvs|^2)R_s(d\vs,d\tvs)\\
&+C(1+||\tf_s||_\infty)  \int_{\rr^3\times\rr^3}\!\! 
\Psi(|v-\tv|^2)Q_s(dv,d\tv).\\
\end{align*}
Due to Remark \ref{conc} and the Jensen inequality,
\begin{align*}
A_s\leq C(1+||f_s+\tf_s||_\infty)\left\{
\Psi\left(\int_{\rr^3\times\rr^3}|v-\tv|^2 Q_s(dv,d\tv)\right)
+\Psi\left(\int_{\rr^3\times\rr^3}|\vs-\tvs|^2 R_s(d\vs,d\tvs)\right)
\right\}.
\end{align*}
We now set $\rho(t):= \E[|V_s-\tV_s|^2]$. Since
$$\cW^2_2(f_s,\tf_s)=\int_{\rr^3\times\rr^3}|\vs-\tvs|^2
R_s(d\vs,d\tvs)\leq
\int_{\rr^3\times\rr^3}|v-\tv|^2
Q_s(dv,d\tv) = \rho(t),
$$ 
and since $\Psi$ is increasing, we have shown that
\begin{equation*}
 \cW^2_2(f_t,\tf_t)\leq \rho(t) \leq 
\cW^2_2(f_0,\tf_0) + C\intot (1+||f_s+\tf_s||_\infty)
\Psi\left(\rho(s)\right)ds.
\end{equation*}
It only remains to check that $\rho$ is bounded
on $[0,T]$. But 
$\rho(t)\leq 2\E[|V_t|^2]+2\E[|\tV_t|^2] =2m_2(f_t)+2m_2(\tf_t)$ 
by Proposition \ref{prop:edslin}-(ii) and $(f_t)_{t\in[0,T]},
(\tf_t)_{t\in[0,T]} \in L^\infty([0,T],\cP_2)$
by assumption.
\end{preuve}

It remains to give the

\begin{preuve} {\it of Proposition \ref{prop:edslin}.}
We only check the results for (\ref{eq:eds}), the study of (\ref{eq:edstilde}) 
being the same.

\vip

{\it Step 1.} 
For $x_0\in\rr^3$ and for $X=(X_t)_{t\in [0,T]}$ a $\rr^3$-valued
progressively measurable
process, we introduce the $\rr^3$-valued progressively measurable process
$(\Phi(x_0,X)_t)_{t\in[0,T]}$ defined by
\begin{equation*}
\Phi(x_0,X)_t=x_0+
\intdd \sigma(X_s-v)W(dv,d\tilde v,ds)+\intd b(X_s-v)f_s(v)dvds.
\end{equation*}
The goal of this step is to prove that $(\Phi(x_0,X)_t)_{t\in{[0,T]}}$
is automatically continuous and that
\begin{equation*}
\E\left[\sup_{[0,T]} |\Phi(x_0,X)_t|^2\right] \leq C \left(|x_0|^2
+\int_0^T ||f_s||_\infty ds+ \left(\int_0^T ||f_s||_\infty ds\right)^2 \right),
\end{equation*}
which is finite thanks to the conditions imposed $f$.
We observe, since the first marginal of $R_s$ is $f_s$ and 
using (\ref{a1}) with $\alpha=-1$, that a.s.,
\begin{align*}
\int_0^T \int_{\rr^3\times\rr^3} |\sigma(X_s-v)|^2 R_s(dv,d\tv)ds
\leq&  \int_0^T \int_{\rr^3} |X_s-v|^{-1} f_s(v)dvds
\leq \int_0^T C(1+||f_s||_\infty)ds <\infty.
\end{align*}
From (\ref{a1}) with $\alpha=-2$,
\begin{align*}
\int_0^T \int_{\rr^3} |b(X_s-v)| f_s(v)dvds
\leq&  \int_0^T \int_{\rr^3} |X_s-v|^{-2} f_s(v)dvds
\leq \int_0^T C(1+||f_s||_\infty)ds <\infty.
\end{align*}
The a.s. continuity of $\Phi(x_0,X)$ on $[0,T]$ follows and the mean square
estimate is easily deduced from the Doob inequality.

\vip

{\it Step 2.} 
We now aim to show that for $t\in[0,T]$, for $X,Y$ two 
progressively measurable processes,
\begin{align*}
\Delta_t := &\E\left[|\Phi(x_0,X)_t-\Phi(x_0,Y)_t|^2\right] \\
\leq& 
C \intot (1+||f_s||_\infty) \left\{\Psi(\E[|\Phi(x_0,X)_s-\Phi(x_0,Y)_s|^2])
+ \Psi(\E[|X_s-Y_s|^2]) \right\}ds,
\end{align*}
where $\Psi$ was defined in (\ref{dfpsi}). Using the It\^o formula
and taking expectations, we derive
\begin{align*}
\Delta_t=&\sum_{i,l=1}^3 \intdd 
\E\big[ \big(\sigma_{il}(X_s-v)-\sigma_{il}( Y_s- v)\big)^2\big]
R_s(dv,d\tilde v) ds \\
&+2\intdd \E\big[ \big(b(X_s-v)-b(Y_s-v)\big).
(\Phi(x_0,X)_s- \Phi(x_0,Y)_s)\big]f_s(v)dv ds.\\
\end{align*}
Due to  (\ref{ttl1}) and (\ref{ttl2}) and since the first marginal of $R_s$
is $f_s$,
\begin{align*}
\Delta_t
\leq& C \intot
\E\left[\int_{\rr^3} |\sigma(X_s-v)-\sigma(Y_s-v)|^2f_s(v)dv\right] ds\\
&+ C\intot  \E\left[|\Phi(x_0,X)_s- \Phi(x_0,Y)_s|\int_{\rr^3} 
|b(X_s-v)-b(Y_s-v)| f_s(v)dv \right] ds\\
\leq& C\intot  (1+||f_s||_\infty)
\E\left[ \Psi(|X_s-Y_s|^2)+ |\Phi(x_0,X)_s- \Phi(x_0,Y)_s|\Psi(|X_s-Y_s|)
\right] ds.
\end{align*}
But since  $x\mapsto \Psi(x)$ is non-decreasing
and $x\Psi(x) \leq \Psi(x^2)$, one has, for all $u,v\geq 0$,
\begin{align*}
u \Psi(v) \leq& \indiq_{\{u\leq v\}} v\Psi(v) + \indiq_{\{u\geq v\}} u\Psi(u)
\leq \Psi(u^2)+\Psi(v^2).
\end{align*}
We thus obtain
\begin{align*}
\Delta_t \leq& C\intot  (1+||f_s||_\infty)
\E\left[ \Psi(|X_s-Y_s|^2)+ \Psi(|\Phi(x_0,X)_s- \Phi(x_0,Y)_s|^2)\right]
ds\\
\leq&  C\intot  (1+||f_s||_\infty)
\left\{ \Psi(\E[|X_s-Y_s|^2])+ \Psi(\E[|\Phi(x_0,X)_s- \Phi(x_0,Y)_s|^2])
\right\}
ds,
\end{align*}
the last inequality following from Remark \ref{conc} 
and the Jensen inequality.

\vip
{\it Step 3.} We now check the uniqueness for (\ref{eq:eds}).
Consider two solutions $V=\Phi(V_0,V)$ and $\tV=\Phi(V_0,\tV)$,
and set $\rho(t)=\E[|V_t-\tV_t|^2]$, which is bounded on $[0,T]$ due
to Step 1. Using Step 2, we deduce that 
$$
\rho(t)\leq  \int_0^t
\gamma(s) \Psi(\rho(s))ds,$$ 
where
$\gamma(s)=C(1+||f_s||_\infty) \in L^1([0,T])$. Lemma \ref{yudo} yields
that $\rho(t)=0$, whence $V_t=\tV_t$ a.s., for all $t\in [0,T]$.
The continuity obtained in Step 1 guarantees us that
a.s., $(V_t)_{t\in [0,T]}=(\tV_t)_{t\in [0,T]}$.

\vip

{\it Step 4.} We now prove the existence of
a solution to (\ref{eq:eds}) using a Picard iteration. We define 
$V^0$ by $V^0_t=V_0$ and then by induction $V^{n+1}=\Phi(V_0,V^n)$.
We then set $\rho_{n,k}(t)=\sup_{[0,t]}\E[|V^{n+k}_s-V^n_s|^2]$, 
which is uniformly
bounded on $[0,T]$ due to Step 1. Step 2 yields 
$$
\rho_{n+1,k}(t) \leq \int_0^t \gamma(s) [ \Psi(\rho_{n+1,k}(s))
+ \Psi(\rho_{n,k}(s)) ]ds, 
$$ 
where 
$\gamma(s)=C(1+||f_s||_\infty) \in L^1([0,T])$. Thus for
$\rho_n(t):=\sup_k \rho_{n,k}(t)$, we get $\rho_{n+1}(t) \leq \int_0^t 
\gamma(s) [\Psi(\rho_{n}(s))+\Psi(\rho_{n+1}(s)) ] ds$. Finally, setting
$\rho(t):=\limsup_n \rho_{n}(t)$, we deduce that $\rho(t) \leq 2\int_0^t 
\gamma(s) \Psi(\rho(s)) ds$, whence $\rho(T)=0$
by Lemma \ref{yudo}. We have proved that
$$
\limsup_n \; \sup_k \sup_{[0,T]}\E[|V^{n+k}_t-V^n_t|^2]=0,
$$
so that the sequence
$(V^n_t)_{t\in [0,T]}$ is Cauchy in $L^\infty([0,T],L^2(\Omega))$. Hence
there is a process $(V_t)_{t\in [0,T]}$ such that 
$\lim_n \sup_{[0,T]} \E[|V_t-V^n_t|^2]=0$. To conclude this step,
it suffices to prove that
$\kappa_n(t):=\E[|\Phi(V_0,V^n)_t-\Phi(V_0,V)_t|^2]$ tends to $0$
for each $t\in [0,T]$. This will allow us to pass to the limit in 
$V^{n+1}=\Phi(V_0,V^n)$ an to get $V=\Phi(V_0,V)$, so that $V$ will solve
(\ref{eq:eds}). The a.s. continuity of $V$ will then be deduced from Step 1.

We set $\e_n:=\sup_{[0,T]} \E[\Psi(|V^n_s-V_s|^2)]$, which tends to $0$ 
by Remark \ref{conc} and the Jensen inequality. 
Using Step 2 again, we immediately obtain, 
for $t\in [0,T]$,
$$\kappa_n(t) \leq \intot \gamma(s)(\e_n + \Psi(\kappa_n(s))) ds.
$$
For $\kappa(t)=\limsup_n \kappa_n(t)$ (which is bounded due to Step 1), 
we get 
$\kappa(t) \leq \intot \gamma(s)\Psi(\kappa(s)) ds$. Lemma \ref{yudo}
thus yields $\kappa(t)=0$ for all $t\in [0,T]$, which concludes this step.

\vip

{\it Step 5.} It remains to prove that for all $s\in [0,T]$,
$\cL(V_s)=f_s$, for $V$ the unique solution of 
(\ref{eq:eds}). We set $g_s=\cL(V_s)$ for all $s\in [0,T]$ and we 
first observe that $(g_t)_{t\in [0,T]}$ solves the 
\emph{linear Landau equation}: for any $\varphi\in C^2_b(\rr^3)$,
\begin{equation}\label{eq:lineq}
\int_{\rr^3}\varphi(x)g_t(dx)= \int_{\rr^3}\varphi(x)f_0(dx)+
\intot \iint_{\rr^3\times \rr^3} L\varphi(x,v) g_s(dx)f_s(v)dv ds .
\end{equation}
with $L$ defined by (\ref{defL}). It suffices to apply 
the It\^o formula, to take expectations, to use that
the first marginal of $R_s$ is $f_s$ and that $\sigma.\sigma^t=a$.
See \cite[around Eq. (2.4)]{fgl} for the detailed
computation.

\vip

Assume for a moment that there is uniqueness
for the linear equation \ref{eq:lineq}. Since $(f_t)_{t\in [0,T]}$ 
is a weak solution to
(\ref{eqlandau}), is also a weak solution to (\ref{eq:lineq}). 
We deduce that for all $t\in [0,T]$, $g_t=f_t$.

\vip

To prove the uniqueness for (\ref{eq:lineq}), we use
a result of Horowitz-Karandikar \cite[Theorem B.1]{hk},
see also Bath-Karandikar \cite[Theorem 5.2]{bk}. Consider,
for $t\in[0,T]$, $x\in\rr^3$ and $\varphi \in C^2_b$,
the operator $\cA_t \varphi(x):= \int_{\rr^3} L\varphi(x,v) f_t(v)dv$.
A stochastic process $(X_t)_{t\in [t_0,T]}$ is said to solve 
the martingale problem for $(C^2_b,\cA_t)$ if for all $\varphi \in C^2_b$, 
the process $\varphi(X_t)-\int_{t_0}^t \cA_s\varphi(X_s)ds$ defined
for $t\in[t_0,T]$ is a martingale.
To apply \cite[Theorem B.1]{hk}, we have to check that:

\noindent (i) there is a countable family 
$(\varphi_k)_{k\geq 1}\subset C^2_b$ such that
for all $t\in [0,T]$, $\{(\varphi_k,\cA_t \varphi_k)\}_{k\geq 1}$ 
is dense in $\{(\varphi,\cA_t\varphi), \varphi\in C^2_b\}$
for the bounded-pointwise convergence;

\noindent (ii) for any $(t_0,x_0)$ in $[0,T]\times \rr^3$, there exists a unique
(in law) solution $(X_t)_{t\in [t_0,T]}$ 
to the martingale problem for $(C^2_b,\cA_t)$ such that $X_{t_0}=x_0$.

\vip

We now verify these two points.
First consider a countable family of functions 
$(\varphi_k)_{k\geq 1}\subset C^2_b$, dense in $C^2_b$ for the norm
$|||\varphi|||:=||\varphi||_{\infty}+||D \varphi||_\infty+
||D^2\varphi||_{\infty}$. Then point (i) easily follows (with the uniform
convergence instead of the bounded-pointwise convergence)
from the estimate
$|\cA_t \varphi (x)| \leq \int_{\rr^3} |L\varphi(x,v)|f_t(v)dv
\leq C |||\varphi||| \int_{\rr^3} (|x-v|^{-1}+|x-v|^{-2})f_t(v)dv \leq C  
|||\varphi|||(1+||f_t||_\infty)$ due to (\ref{a1}) with $\alpha=-1$ 
and $\alpha=-2$.
To prove (ii), observe 
that the martingale problem for $(C^2_b,\cA_t)$ with $X_{t_0}=x_0$
corresponds to the stochastic differential equation
$$
X_t=x_0 + \int_{t_0}^t \int_{\rr^3\times\rr^3} 
\sigma(X_s-v)W(dv,d\tv,ds) + \int_{t_0}^t \int_{\rr^3} 
b(X_s-v)f_s(dv)ds, 
$$
for which we have shown the strong existence and uniqueness (only in
the case $t_0=0$ and $x_0=V_0$, 
but the generalization is straightforward). Point (ii) follows. 
\end{preuve}

\section{Conclusion}\label{concl}

Our main result is easily deduced from Theorem \ref{fonda}
and Lemma \ref{yudo}.

\begin{preuve} {\it of Theorem \ref{mr}.}
Let $T>0$ be fixed.

\vip

{\it Point (i).}
Assume that we have two weak solution
$(f_t)_{t\in[0,T]},(\tf_t)_{t\in[0,T]}$ to (\ref{eqlandau}), with $\tf_0=f_0$,
both belonging to $L^\infty([0,T],\cP_2)\cap
L^1([0,T],L^\infty)$. Theorem \ref{fonda} implies that there
is a bounded function $\rho:[0,T]\mapsto \rr_+$ such that
$\cW^2_2(f_t,\tf_t) \leq \rho(t)$ for all $t\in [0,T]$ satisfying
$\rho(t) \leq \int_0^t \gamma(s) \Psi(\rho(s))ds$,
with $\gamma(s)=C(1+||f_s+\tf_s||_\infty)\in L^1([0,T])$. 
Lemma \ref{yudo}-(i) implies
that $\rho(t)=0$, whence $\cW_2(f_t,\tf_t)=0$, for all
$t\in [0,T]$. Thus $(f_t)_{t\in[0,T]}=(\tf_t)_{t\in[0,T]}$.

\vip

{\it Point (ii).} Let now $(f_t)_{t\in[0,T]},(f^n_t)_{t\in[0,T]}$ be 
a family of weak solutions such that 
$R_T:=\sup_n \int_0^T (1+||f^n_s+f_s||_\infty) ds <\infty$ and
$a_n:=\cW_2^2(f^n_0,f_0)\to 0$. Applying Theorem \ref{fonda}
we get $\cW_2^2(f^n_t,f_t)\leq \rho_n(t)$, for some bounded function
$\rho_n:[0,T]\mapsto \rr_+$ satisfying 
$$
\rho_n(t)\leq a_n + \int_0^t C (1+ ||f^n_s+f_s||_\infty) \Psi(\rho_n(s))ds.
$$
Lemma \ref{yudo}-(ii) implies $M(a_n)-M(\rho_n(t)) \leq CR_T$,
where $M(x)=\int_x^1 (1/\Psi(y))dy$ is decreasing on 
$(0,\infty)$ and satisfies $\lim_{x \searrow 0}M(x)=+\infty$.
As a consequence, 
$$
\liminf_n M(\sup_{[0,T]}\rho_n(t))=
\liminf_n \inf_{[0,T]} M(\rho_n(t)) \geq \liminf_n M(a_n)-
CR_T=+\infty. 
$$
This implies that $\lim_n \sup_{[0,T]} \rho_n(t)=0$,
and finally, $\lim_n \sup_{[0,T]} \cW^2_2(f^n_t,f_t)=0$.
\end{preuve}

\end{document}